\documentclass[reqno,12pt]{amsart}
\usepackage{geometry}
\usepackage{cite}
\usepackage{epsfig}
\usepackage{amsmath}
\usepackage{amsfonts}
\usepackage{mathrsfs}
\usepackage{subfigure}
\usepackage[colorlinks,
            linkcolor=blue,
            anchorcolor=blue,
            citecolor=blue
            ]{hyperref}

\numberwithin{equation}{section}
\allowdisplaybreaks[4]

\begin{document}
\newtheorem{theorem}{Theorem}[section]
\newtheorem{corollary}[theorem]{Corollary}
\newtheorem{proposition}[theorem]{Proposition}
\newtheorem{remark}[theorem]{Remark}

\newcommand{\s}{\sum_{n=0}^{\infty}}
\newcommand{\su}{\sum_{n=1}^{\infty}}
\newcommand{\D}{\mathcal{D}}

\def\square{\hfill${\vcenter{\vbox{\hrule height.4pt \hbox{\vrule width.4pt
height7pt \kern7pt \vrule width.4pt} \hrule height.4pt}}}$}

\title[]
{\small Two $q$-operational equations and Hahn polynomials}

%%the first author
\author{Jing Gu$^{1}$}

\address{Jing Gu$^{1}$, $^{1}$School of Mathematical Sciences, East China Normal University, Shanghai, 200241, China}

\email{52215500005@stu.ecnu.edu.cn}

%% Corresponding author
\author{DunKun Yang$^{1}$}

\address{DunKun Yang$^{1}$, $^{1}$School of Mathematical Sciences, East China Normal University, Shanghai, 200241, China}

\email{52265500003@stu.ecnu.edu.cn}

\author{Qi Bao$^{1,\ast}$}

\address{Qi Bao$^{1}$, $^{1}$School of Mathematical Sciences, East China Normal University, Shanghai, 200241, China}

\email{52205500010@stu.ecnu.edu.cn}

\subjclass[2010]{05A30, 33D90}

\keywords{Basic hypergeometric series; $q$-derivative; $q$-operational equation; Hahn polynomials; $q$-Gaussian summation}

\thanks{$^{\ast}$Corresponding author. Email address: 52205500010@stu.ecnu.edu.cn}

\begin{abstract}
Motivated by Liu's recent work in \cite{Liu2022}. We shall reveal the essential feature of Hahn polynomials by presenting two new $q$-exponential operators. These lead us to use a systematic method to study identities involving  Hahn polynomials. As applications, we use the method of $q$-exponential operator to prove the bilinear generating function of Hahn polynomials and Heine's second transformation formula. Moreover, a generalization of $q$-Gaussian summation is given, too.
\end{abstract}

\maketitle
%section1
\section{Introduction}
The theory of basic hypergeometric series (or $q$-series) is one of the important branches of mathematics. It is well known that it has applications in various branches, such as combinatorial mathematics, number theory, computational algebra, orthogonal polynomial theory, difference equations, algebraic geometry, Lie algebra, Lie groups, statistics and physics, see \cite{A1986,F1988}. After nearly a hundred years of systematic development, mathematicians have studied it using various methods, including the Wilf-Zeilberg algorithm, transformation, inversion, operator and $q$-partial differential equation (cf. \cite{WZ1992,Chu1993,GR,VS1999,AA1993,Liu2013,CZA,AHD1977,A1970,GS1983}). In 1997, Chen and Liu \cite{ChenLiu1997,ChenLiu1998} first used the method of $q$-exponential operators to study basic hypergeometric series. Many classical $q$-series results fall into this theory, such as the Askey-Wilson integral, the Nassrallah-Rahman integral, the $q$-integral forms of Sears transformation, and Gasper's formula of the extension of the Askey-Roy integral, etc. Meanwhile, many new $q$-series identities are obtained. Since then, more and more mathematicians have used this powerful method to study basic hypergeometric series, see \cite{Zhang2020,A2016,LZ2015,SA2016,CZ2014,Jia2014,Cao2014,Cao2014a,Liu2011,
ZY2010,Lu2009,CG2008,F2007,ZW2005,CSS} for more details.

Now let's give some standard notations (cf. \cite{GR}). Let $q$ be a complex number such that $|q|<1$. For any complex number $a$, the $q$-shifted factorials $(a;q)_n$ are defined by
\begin{align*}
(a;q)_0=1, \,\, (a;q)_n=\prod_{k=0}^{n-1}(1-aq^k), \,\,
(a;q)_{\infty}=\prod_{k=0}^{\infty}(1-aq^k).
\end{align*}
For convenience, we also adopt the following compact notation for the multiple $q$-shifted factorial
\begin{align*}
(a_1,a_2,\cdots,a_m;q)_n=(a_1;q)_n(a_2;q)_n\cdots (a_m;q)_n,
\end{align*}
where $n$ is a nonnegative integer or positive infinity. The basic hypergeometric series or $q$-hypergeometric series ${}_r\phi_s(\cdot)$ are defined by
\begin{align}\label{rs}
{}_r\phi_s \left( \begin{gathered}
a_1, a_2, \cdots , a_r \\ b_1, b_2, \cdots , b_s \end{gathered}
;\,q, z \right)=\s \frac{ (a_1,  \cdots a_r;q)_n  }
{ (q, b_1, \cdots b_s;q)_n}
\left[ (-1)^n q^{ \binom n2 } \right]^{1+s-r} z^n.
\end{align}
Here and in what follows, $\binom n2$ stands for composite symbol. For any function $f(x)$ of one variable, the $q$-derivative of $f(x)$ with respect to $x$ is defined as $\D \{f(x)\}=(f(x)-f(qx))/x$. A $q$-partial derivative of a function of several variables is its
$q$-derivative with respect to one of those variables, regarding other variables as constants (cf. \cite{Liu2013}). For convenience, the $q$-partial derivative of a function $f$ with respect to the variable $x$ is denoted by $\D_x\{f\}$. The famous $q$-binomial theorem as follows
\begin{align}\label{binomialTh}
{}_1\phi_0\left( \begin{gathered} a \\ - \end{gathered} ;\,q, z \right)
=\s \frac{(a;q)_n}{(q;q)_n}z^n
= \frac{(az;q)_{\infty}}{(z;q)_{\infty}}, \,|z|<1,
\end{align}
which can derive the following two identities
\begin{align}\label{binomialTh1}
\exp_q(z)\equiv \s \frac{z^n}{(q;q)_n}
= \frac{1}{(z;q)_{\infty}},\, |z|<1,  \quad
\s \frac{(-1)^n q^{\binom n2}}{(q;q)_n}z^n= (z;q)_{\infty}.
\end{align}
The famous Jackson's transformation formula is as follows (cf. \cite[(1.5.4)]{GR})
\begin{align}\label{1eq3}
{}_2\phi_1\left( \begin{gathered} a, b \\ c \end{gathered}
;\,q, z \right)=\frac{(az;q)_{\infty}}{(z;q)_{\infty}}
{}_2\phi_2\left( \begin{gathered} a, c/b \\ c, az \end{gathered}
;\,q, bz \right), \, \max \{|z|, |bq|\}<1.
\end{align}
For any real number $r$, the $q$-shift operator $\eta_{x_i}^r$ for a function $f(x_1, x_2,\cdots, x_n)$ is defined by
\begin{align*}
\eta_{x_i}^r \{ f \}=f(x_1,\cdots,x_{i-1},q^r x_i,x_{i+1},\cdots,x_n).
\end{align*}
Very recently, Liu proved \cite{Liu2022} the following $q$-exponential operational identity.
\begin{theorem}\label{Th}
If $f(x)$ is a function of $x$, then, under suitable convergence conditions, we have the following $q$-exponential operational identity
\begin{align*}
\exp_q(t\Delta_x)f(x)=\frac{1}{(xt;q)_{\infty}}
\s \frac {t^n}{(q;q)_n} f(q^n x),
\end{align*}
where operational $\Delta_x\equiv x+\eta_x$.
\end{theorem}
This operational identity reveals an essential feature of the Rogers-Szeg\H{o} polynomials. Therefore Liu developed a systematic method to prove the identities involving the Rogers-Szeg\H{o} polynomials in \cite{Liu2022}. Motivated by Theorem \ref{Th}. This section mainly generalizes Theorem \ref{Th}. For that, we construct two operators $\Delta_{x,a}$ and $\Omega_{x,a}$ as follow
\begin{align*}
\Delta_{x,a}=x (1-a)\eta_a+\eta_x \quad {\rm and} \quad
\Omega_{x,a}=x+(1-a)\eta_a \eta_x.
\end{align*}
It is worth pointing out that both $\Delta_{x,a}$ and $\Omega_{x,a}$ degenerate to $\Delta_x$ when $a=0$. We shall establish the following theorem.
\begin{theorem}\label{Th1}
If $f(x)$ is a function of $x$, then, under suitable convergence conditions, we have the following two $q$-exponential operational identities:
\begin{align*}
\exp_q(t\Delta_{x,a})f(x)&=\frac{(axt;q)_{\infty}}{(xt;q)_{\infty}}
\s \frac {t^n}{(q;q)_n} f(q^n x), \\
\exp_q(t \Omega_{x,a})f(x)&=\frac{1}{(xt;q)_{\infty}} \s \frac{(a;q)_n t^n}{(q;q)_n} f(q^n x).
\end{align*}
\end{theorem}

\begin{proof}
For the first equation, define the function $F(a,x,t)$ by
\begin{align}\label{1eq2}
F(a,x,t)\equiv \frac{(axt;q)_{\infty}}{(xt;q)_{\infty}}
\s \frac {t^n}{(q;q)_n} f(q^n x).
\end{align}
We know that $F(a,x,t)$ is a product of two formal power series of $t$, thus $F(a,x,t)$ is a formal power series of $t$. By simple calculation,
\begin{align*}
F(a,x,qt)&=\frac{(aqxt;q)_{\infty}}{(qxt;q)_{\infty}}
\s \frac {t^{n}q^{n}}{(q;q)_n} f(q^n x) \\
&=\frac{(aqxt;q)_{\infty}}{(qxt;q)_{\infty}}
\s \frac{t^n}{(q;q)_n} f(q^n x)
-\frac{(aqxt;q)_{\infty}}{(qxt;q)_{\infty}}
\su \frac {t^n}{(q;q)_{n-1}} f(q^n x) \\
&=\frac{(aqxt;q)_{\infty}}{(qxt;q)_{\infty}}
 \s \frac{ t^n}{(q;q)_n} f(q^n x)
-\frac{t(aqxt;q)_{\infty}}{(qxt;q)_{\infty}}
\s \frac{ t^n}{(q;q)_n} f(q^{n+1} x).
\end{align*}
Further, we have
\begin{align*}
&F(a,x,t)-F(a,x,qt) \\
&=\frac{xt(1-a)(aqxt;q)_{\infty}}{(xt;q)_{\infty}}
\s \frac{ t^n}{(q;q)_n} f(q^n x)+\frac{t(aqxt;q)_{\infty}}{(qxt;q)_{\infty}}
\s \frac{ t^n}{(q;q)_n} f(q^{n+1} x)  \\
&=xt(1-a)\eta_a F(a,x,t)+t \eta_x  F(a,x,t).
\end{align*}
It follows from which that
\begin{align*}
\D_t F(a,x,t)=(x (1-a)\eta_a+\eta_x) F(a,x,t)=\Delta_{x,a} F(a,x,t).
\end{align*}
Now, we start to solve the following operator equations
\begin{align}\label{1eq1}
\begin{cases}
\D_t F(a,x,t)=\Delta_{x,a} F(a,x,t),\\
F(a,x,0)=f(x).
\end{cases}
\end{align}
By \cite[Theorems 1 and 2]{Liu2010}, we can assume that
\begin{align*}
F(a,x,t)=\s A_n(a,x) \frac{t^n}{(q;q)_n}.
\end{align*}
Taking above equation into the first equation in (\ref{1eq1}), we deduce that
\begin{align*}
\su A_n(a,x) \frac{t^{n-1}}{(q;q)_{n-1}}
=\s \Delta_{x,a} A_n(a,x) \frac{t^n}{(q;q)_n}.
\end{align*}
Comparing the coefficients of $t^{n-1}$ in the above equation, we obtain $A_n(a,x)=\Delta_{x,a} A_{n-1}(a,x)$. It is worth noting that $A_0(a,x)=F(a,x,0)=f(x)$, then $A_n(a,x)=\Delta_{x,a}^n f(x)$. Finally, our conclusion is
\begin{align*}
F(a,x,t)=\s \frac{t^n}{(q;q)_n} \Delta_{x,a}^n f(x)
=\exp_q(t\Delta_{x,a})f(x).
\end{align*}
Combining this equation and (\ref{1eq2}), we conclude that
\begin{align*}
\exp_q(t\Delta_{x,a})f(x)&=\frac{(axt;q)_{\infty}}{(xt;q)_{\infty}}
\s \frac {t^n}{(q;q)_n} f(q^n x),
\end{align*}
this completes the first part of the proof. The proof of the second equation is similar to the first one. We set
\begin{align}\label{1eq4}
G(a,x,t)\equiv \frac{1}{(xt;q)_{\infty}}
\s \frac{(a;q)_n t^n}{(q;q)_n} f(q^n x).
\end{align}
Clearly, $G(a,x,t)$ is also a formal power series of $t$, by simple calculation, we have
\begin{align*}
G(a,x,qt)&=\frac{1}{(qxt;q)_{\infty}}
\s \frac{(a;q)_n t^{n}q^{n}}{(q;q)_n} f(q^n x) \\
&=\frac{1}{(qxt;q)_{\infty}} \s \frac{(a;q)_n t^n}{(q;q)_n} f(q^n x)
-\frac{1}{(qxt;q)_{\infty}} \su \frac{(a;q)_n t^n}{(q;q)_{n-1}} f(q^n x)\\
&=\frac{1}{(qxt;q)_{\infty}} \s \frac{(a;q)_n t^n}{(q;q)_n} f(q^n x)
-\frac{(1-a)t}{(qxt;q)_{\infty}} \s \frac{(aq;q)_n t^n}{(q;q)_n} f(q^{n+1}x).
\end{align*}
Therefore,
\begin{align*}
&G(a,x,t)-G(a,x,qt)\\
&=\frac{xt}{(xt;q)_{\infty}} \s \frac{(a;q)_n t^n}{(q;q)_n} f(q^n x)
+\frac{(1-a)t}{(qxt;q)_{\infty}}\s\frac{(aq;q)_n t^n}{(q;q)_n}f(q^{n+1}x)\\
&=xt G(a,x,t)+t(1-a)\eta_x \eta_a G(a,x,t).
\end{align*}
It follows from which that
\begin{align*}
\D_t G(a,x,t)=(x+(1-a)\eta_x \eta_a) G(a,x,t)=\Omega_{x,a} G(a,x,t).
\end{align*}
By solving the following system of operators,
\begin{align*}
\begin{cases}
\D_t G(a,x,t)=\Omega_{x,a} G(a,x,t),\\
G(a,x,0)=f(x),
\end{cases}
\end{align*}
One can easily obtain that
\begin{align*}
G(a,x,t)=\s \frac{t^n}{(q;q)_n} \Omega_{x,a}^n f(x)
=\exp_q(t\Omega_{x,a})f(x).
\end{align*}
Combining this equation and (\ref{1eq4}), we conclude that
\begin{align*}
\exp_q(t\Omega_{x,a})f(x)
=\frac{1}{(xt;q)_{\infty}}
\s \frac{(a;q)_n t^n}{(q;q)_n} f(q^n x).
\end{align*}
which completed the proof.
\end{proof}
This paper is organized as follows. In section \ref{Sec2}, we mainly present some properties of operators $\Delta_{x,a}$ and $\Omega_{x,a}$. Section \ref{Sec3} will use these two operators to derive several identities involving homogeneous Hahn polynomials (see (\ref{q-Hahn}) for definition). Section \ref{Sec4} is the applications of these two operators. In the last section, we will give a generalization of the $q$-Gaussian summation formula.

\section{Properties of operators $\Delta_{x,a}$ and $\Omega_{x,a}$}\label{Sec2}

The Hahn polynomials are $q$-orthogonal polynomials whose applications and generalizations arise in many applications such as the $q$-harmonic oscillator, theta functions, quantum groups, and coding theory. They were first studied by Hahn, and then by Al-Salam and Carlitz
(cf.\cite{Hahn,Al1965}). So they are also called Al-Salam-Carlitz polynomials. In \cite{L2015}, Liu studied the homogeneous Hahn polynomials from the perspective of $q$-partial differential equation, which are defined as
\begin{align}\label{q-Hahn}
\Phi_n^{(a)}(x,y|q)=\sum_{k=0}^n \begin{bmatrix} n \\k \end{bmatrix}
(a;q)_k x^k y^{n-k},
\end{align}
where
\begin{align*}
\begin{bmatrix} n \\k \end{bmatrix}
=\frac{ (q;q)_n }{ (q;q)_k (q;q)_{n-k} },
\,\, 0\leq k \leq n.
\end{align*}
The above equation is called Gaussian binomial coefficients are $q$-analogs of the binomial coefficients. When $y=1$, homogeneous Hahn polynomials degenerates to original Hahn polynomials $\Phi_n^{(a)}(x|q)$, that is, $\Phi_n^{(a)}(x,1|q)=\Phi_n^{(a)}(x|q)$.

This section will give some properties of operators $\Delta_{x,a}$ and $\Omega_{x,a}$. Moreover, we will establish the relationship between these two operators and homogeneous Hahn polynomials. First of all, taking $f(x)\equiv1$ in Theorem \ref{Th1} and combining separately with (\ref{binomialTh}) and the first identity of (\ref{binomialTh1}) to immediately leads to the following Proposition \ref{proposition1}.
\begin{proposition}\label{proposition1}
For $\max\{ |t|,|xt|\}<1$, we have
\begin{align*}
\exp_q(t\Delta_{x,a})1=\frac{(axt;q)_{\infty}}{(t,xt;q)_{\infty}}, \quad
\exp_q(t\Omega_{x,a})1=\frac{(at;q)_{\infty}}{(t,xt;q)_{\infty}}.
\end{align*}
\end{proposition}

\begin{proposition}
For $\max\{ |t|,|s|,|xt|,|xs|\}<1$ and $a\neq b$, we have
\begin{align}
\exp_q(s\Delta_{x,b}) \exp_q(t\Delta_{x,a})1
&=\frac{(axt,bxs;q)_{\infty}}{(t,s,xt,xs;q)_{\infty}}
{}_1\phi_1\left( \begin{gathered} a \\ axt \end{gathered}
;\,q, xts \right), \label{2eq1} \\
\exp_q(s\Omega_{x,b}) \exp_q(t\Omega_{x,a})1
&=\frac{(at,bs;q)_{\infty}}{(t,s,xt,xs;q)_{\infty}}
{}_1\phi_1\left( \begin{gathered} b
\\ bs \end{gathered} ;\,q, xts \right).  \label{2eq2}
\end{align}
\end{proposition}

\begin{proof}
Appealing to  Proposition \ref{proposition1} and the first identity of Theorem \ref{Th1}, we obtain
\begin{align*}
\exp_q(s\Delta_{x,b}) \exp_q(t\Delta_{x,a})1
=&\exp_q(s\Delta_{x,b}) \left\{ \frac{(axt;q)_{\infty}}
  {(t,xt;q)_{\infty}} \right\}\\
=&\frac{(bsx;q)_{\infty}}{(t,xt;q)_{\infty}} \s \frac{s^n}{(q;q)_n} \frac{(axtq^n;q)_{\infty}}{(xtq^n;q)_{\infty}}\\
=&\frac{(axt,bxs;q)_{\infty}}{(t,xt,xs;q)_{\infty}}
{}_2\phi_1\left( \begin{gathered} xt, o \\ axt \end{gathered};\,q, s \right).
\end{align*}
Next, making use of (\ref{1eq3}), then we completes the proof of (\ref{2eq1}). Similarly, appealing to  Proposition \ref{proposition1} and the second identity of Theorem \ref{Th1}, we have
\begin{align}\label{eq1}
\exp_q(s\Omega_{x,b}) \exp_q(t\Omega_{x,a})1
&=\exp(s\Omega_{x,b}) \left\{ \frac{(at;q)_{\infty}}{(t,xt;q)_{\infty}} \right\} \nonumber \\
&=\frac{(at;q)_{\infty}}{(t,sx;q)_{\infty}} \s \frac{(b;q)_n}{(q;q)_n} \frac{s^n}{(xtq^n;q)_{\infty}} \nonumber\\
&=\frac{(at;q)_{\infty}}{(t,sx,tx;q)_{\infty}}
\s \frac{(b;q)_n (xt;q)_n}{(q;q)_n} s^n.
\end{align}
Next, using formula (\ref{binomialTh}) and $(xt;q)_n=\sum_{k=0}^n
{\tiny \begin{bmatrix} n \\k \end{bmatrix}}
 q^{\binom k2} (-xt)^k $. Then we obtain
\begin{align*}
\s \frac{(b;q)_n (xt;q)_n}{(q;q)_n} s^n
&=\s \frac{(b;q)_n}{(q;q)_n} s^n \sum_{k=0}^n \begin{bmatrix} n \\k \end{bmatrix} q^{\binom k2} (-xt)^k \\
&=\sum_{k=0}^{\infty} \sum_{n=k}^{\infty}  \frac{(b;q)_n q^{\binom k2} (-xt)^k s^n}{(q;q)_k (q;q)_{n-k} } \\
&=\sum_{k=0}^{\infty} \frac{(b;q)_k q^{\binom k2} (-xts)^k }{(q;q)_k}
\s \frac{(bq^k;q)_n}{(q;q)_n}s^n \\
&=\sum_{k=0}^{\infty} \frac{(b;q)_k q^{\binom k2} (-xts)^k }{(q;q)_k}
 \frac{(bsq^k;q)_{\infty}}{(s;q)_{\infty}}.
\end{align*}
Substituting the above equation into (\ref{eq1}). We complete the proof of (\ref{2eq2}).
\end{proof}

From \cite[Ex. 1.35]{GR} and definitions of $\Delta_{x,a}$ and $\Omega_{x,a}$, we can easily derive the following operational identities. Here we omit the details of the calculation.
\begin{align}\label{eq6}
\Delta_{x,a}^n =\sum_{k=0}^n \begin{bmatrix} n \\k \end{bmatrix} (a;q)_k x^k \eta_a^k \eta_x^{n-k}, \quad
\Omega_{x,a}^n =\sum_{k=0}^n \begin{bmatrix} n \\k \end{bmatrix} (a;q)_k x^{n-k} \eta_a^k \eta_x^k.
\end{align}
If $f$ were only a univariate function of $x$, it follows from (\ref{eq6}) that
\begin{align}
\Delta_{x,a}^n f(x)&=\sum_{k=0}^n \begin{bmatrix} n \\k \end{bmatrix}
(a;q)_k x^k \eta_x^{n-k} f(x), \label{2Eq2} \\
\Omega_{x,a}^n f(x)&=\sum_{k=0}^n \begin{bmatrix} n \\k \end{bmatrix}
(a;q)_k x^{n-k} \eta_x^k f(x).  \label{2Eq2'}
\end{align}
Taking $f(x)\equiv1$ in (\ref{2Eq2})--(\ref{2Eq2'}) and noting the definition of the homogeneous Hahn polynomials in (\ref{q-Hahn}). We arrive at the following operational representations of the homogeneous Hahn polynomials.
\begin{proposition} For any non-negative integer $n$, we have
\begin{align}
\Delta_{x,a}^n 1&=\sum_{k=0}^n \begin{bmatrix} n \\k \end{bmatrix}
(a;q)_k x^k =\Phi_n^{(a)}(x,1|q), \label{2eq4} \\
\Omega_{x,a}^n 1&=\sum_{k=0}^n \begin{bmatrix} n \\k \end{bmatrix}
(a;q)_k x^{n-k}=\Phi_n^{(a)}(1,x|q). \label{2eq4'}
\end{align}
\end{proposition}

\begin{proposition}\label{2eq10}
For any non-negative integers $m$ and $n$, we find that
\begin{align*}
\Delta_{x,a}^n \Phi_m^{(a)}(x,1|q)&=\Phi_{m+n}^{(a)}(x,1|q), \\
\Omega_{x,a}^n \Phi_m^{(a)}(1,x|q)&=\Phi_{m+n}^{(a)}(1,x|q).
\end{align*}
\end{proposition}
\begin{proof}
The proof of the second identity is similar to the first one. We only prove the first identity. Using $\Delta_{x,a}^{n+m}=\Delta_{x,a}^n \Delta_{x,a}^m$ and (\ref{2eq4}), we immediately conclude that
\begin{align*}
\Phi_{m+n}^{(a)}(x,1|q)=\Delta_{x,a}^n \Delta_{x,a}^m 1
=\Delta_{x,a}^n \Phi_m^{(a)}(x,1|q).
\end{align*}
\end{proof}

\begin{proposition}\label{proposition3}
For $a\neq b$, we have the following operational representations
\begin{align*}
\exp_q(t \Delta_{x,a} \Delta_{y,b})1&=
\s \Phi_n^{(a)}(x,1|q) \Phi_n^{(b)}(y,1|q) \frac{t^n}{(q;q)_n}, \\
\exp_q(t \Omega_{x,a} \Omega_{y,b})1&=
\s \Phi_n^{(a)}(1,x|q) \Phi_n^{(b)}(1,y|q) \frac{t^n}{(q;q)_n}.
\end{align*}
\end{proposition}
\begin{proof}
It follows from (\ref{2eq4}) that for any non-negative integer $n$,
\begin{align*}
\Delta_{x,a}^n \Delta_{y,b}^n 1=\Delta_{x,a}^n \Phi_n^{(b)}(y,1|q)
=\Phi_n^{(b)}(y,1|q) \Delta_{x,a}^n 1
=\Phi_n^{(a)}(x,1|q) \Phi_n^{(b)}(y,1|q).
\end{align*}
Thereby, taking advantage of the above equation, we have
\begin{align*}
\exp_q(t \Delta_{x,a} \Delta_{y,b})1
=\s \frac{t^n}{(q;q)_n} \Delta_{x,a}^n \Delta_{y,b}^n 1
=\s \frac{t^n}{(q;q)_n} \Phi_n^{(a)}(x,1|q) \Phi_n^{(b)}(y,1|q).
\end{align*}
This completes the proof of the first equation. The proof of the second equation is similar to the first one. We omit it.
\end{proof}
The following Theorems \ref{2th2}-\ref{2eq5} reveal some profound properties of operators $\Delta_{x,a}$ and $\Omega_{x,a}$, which will be applied in later sections.
\begin{theorem}\label{2th2}
If $f(x)$ is a function of $x$, then, under suitable convergence conditions, we have the following $q$-exponential identities
\begin{align*}
\exp_q(t\Delta_{x,a})f(x)&=\frac{(axt;q)_{\infty}}{(t,xt;q)_{\infty}}
\s \frac{ (-xt)^n}{(q;q)_n} q^{\binom n2} \D_x^n f(x), \\
\exp_q(t\Omega_{x,a})f(x)&=\frac{(at;q)_{\infty}}{(t,xt;q)_{\infty}}
\s \frac{(a;q)_n (-xt)^n}{(q;q)_n(at;q)_n} q^{\binom n2} \D_x^n f(x).
\end{align*}
\end{theorem}
\begin{proof}
For any non-negative integer $n$, we can easily obtain the following equation by mathematical induction (cf. \cite[(3.1)]{Liu2022}),
\begin{align}\label{eq2}
f(q^nx)=\sum_{k=0}^n \begin{bmatrix} n \\k \end{bmatrix}
(-1)^k q^{\binom k2} x^k \D_x^k f(x).
\end{align}
It follows from which that
\begin{align}
\frac {t^n}{(q;q)_n} f(q^n x)
&=\s \frac{ t^n}{(q;q)_n} \sum_{k=0}^n \begin{bmatrix} n \\k \end{bmatrix} (-1)^k q^{\binom k2} x^k \D_x^k f(x) \nonumber\\
&=\sum_{k=0}^{\infty} \sum_{n=k}^{\infty}
\frac{ q^{\binom k2} t^n (-x)^k}{(q;q)_k (q;q)_{n-k}} \D_x^k f(x)\nonumber\\
&=\frac{1}{(t;q)_{\infty}}\sum_{k=0}^{\infty}
\frac{ (-xt)^k}{(q;q)_k} q^{\binom k2} \D_x^k f(x),\label{2eq3}
\end{align}
Substituting the above equation into the right-hand side of the first equation in Theorem \ref{Th1}, this completes the proof of the first equation. Applying the same method to the second equation in Theorem \ref{2th2}. Making use of (\ref{eq2}), we obtain that
\begin{align}\label{2eq9}
 \s \frac{(a;q)_n t^n}{(q;q)_n} f(q^n x)
&=\s \frac{(a;q)_n t^n}{(q;q)_n} \sum_{k=0}^n \begin{bmatrix} n \\k \end{bmatrix} (-1)^k q^{\binom k2} x^k \D_x^k f(x) \nonumber\\
&=\sum_{k=0}^{\infty} \sum_{n=k}^{\infty}
\frac{(a;q)_n t^n q^{\binom k2} (-x)^k}{(q;q)_k (q;q)_{n-k}} \D_x^k f(x)\nonumber \\
&=\sum_{k=0}^{\infty} \frac{(a;q)_k q^{\binom k2} (-xt)^k}{(q;q)_k}
\D_x^k f(x) \s \frac{(aq^k;q)_n t^n}{(q;q)_n}.
\end{align}
Substituting (\ref{2eq9}) into the right-hand side of the second equation of Theorem \ref{Th1}, which is equivalent to the right-hand side of the second equation in Theorem \ref{2th2}. This completes the proof of Theorem \ref{2th2}.
\end{proof}

%\begin{theorem}
%we have
%\begin{align}
%\exp_q(t\Omega_{x,a}\Omega_{y,b})f(x,y)=\frac{1}{(xyt;q)_{\infty}}
%\sum_{k,l,m=0}^{\infty} \frac{(a;q)_l (a;q)_k (a;q)_{m+k} x^l y^m q^{lm} t^{l+k+m}}{ (q;q)_l (q;q)_k (q;q)_m } f(q^{k+m}x, q^{k+l}y)
%\end{align}
%\end{theorem}

%\begin{proof}
%\begin{align}
%\exp(t\Omega_{x,a})f(x,y)=\frac{1}{(xt;q)_{\infty}} \s \frac{(a;q)_n t^n}{(q;q)_n} f(q^n x,y).
%\end{align}
%Letting $t\to t\Omega_{y,b}$,
%\begin{align}
%\exp(t\Omega_{x,a}\Omega_{y,b})f(x,y)=\exp(xt\Omega_{y,b})
%\s \frac{(b;q)_n t^n}{(q;q)_n} \Omega_{y,a}^n f(q^n x,y).
%\end{align}
%since
%\begin{align}
%\Omega_{y,b}^n=\sum_{k=0}^n \begin{bmatrix} n \\k \end{bmatrix}
%(b;q)_k y^{n-k} \eta_y^k
%\end{align}
%then we have
%\begin{align}
%\exp(t\Omega_{x,a}\Omega_{y,b})f(x,y)&=\exp(xt\Omega_{y,b})
%\s \frac{(b;q)_n t^n}{(q;q)_n} \sum_{k=0}^n \begin{bmatrix} n \\k \end{bmatrix} (a;q)_k y^{n-k} f(q^n x, q^k y) \\
%&=\exp(xt\Omega_{y,b}) \sum_{k=0}^{\infty} \frac{(b;q)_k t^k}{(q;q)_k}
%\sum_{n=k}^{\infty} \frac{(a;q)_n t^{n-k} y^{n-k}}{(q;q)_{n-k}}
%f(q^n x, q^k y) \\
%&=\exp(xt\Omega_{y,b}) \sum_{k=0}^{\infty} \frac{(b;q)_k t^k}{(q;q)_k}
%\sum_{m=0}^{\infty} \frac{(a;q)_{m+k} (ty)^m }{(q;q)_m}
%f(q^{k+m} x, q^k y)
%\end{align}
%then by
%\begin{align}
%\exp(t\Omega_{x,a}\Omega_{y,b})f(x,y)&=\frac{1}{(xyt;q)_{\infty}}
%\sum_{l=0}^{\infty} \frac{(b;q)_l (xt)^l}{(q;q)_l}
 %\sum_{k=0}^{\infty} \frac{(b;q)_k t^k}{(q;q)_k}
%\sum_{m=0}^{\infty} \frac{(a;q)_{m+k} t^m (yq^l)^m }{(q;q)_m}
%f(q^{k+m} x, q^{k+l} y)
%\end{align}
%\end{proof}

\begin{theorem}\label{th5}
If $f(x)$ is a function of $x$, then, under suitable convergence conditions, we have
\begin{align*}
\Delta_{x,a}^n f(x)=\sum_{k=0}^n \begin{bmatrix} n \\k \end{bmatrix} (-x)^k
q^{\binom k2} \Phi_{n-k}^{(a)}(x,1|q) \D_x^k f(x).
\end{align*}
\end{theorem}

\begin{proof}
By (cf. \cite[(3.1)]{L2015}), we have
\begin{align}\label{2eq12}
\sum_{n=0}^{\infty}\Phi_{n}^{(a)}(x,y|q)
\frac{t^{n}}{(q;q)_{n}}
=\frac{(axt;q)_{\infty}}{(xt,yt;q)_{\infty}}.
\end{align}
Taking $y=1$ in the above identity and by the first equation of Theorem \ref{2th2}, we find that
\begin{align}\label{2eq11}
\s \frac{t^n}{(q;q)_n} \Delta_{x,a}^n f(x)
&=\s \Phi_n^{(a)}(x|q) \frac{t^n}{(q;q)_n} \sum_{k=0}^{\infty}
\frac{ (-xt)^k}{(q;q)_k} q^{\binom k2} \D_x^k f(x) \nonumber\\
&=\s \frac{t^{n}}{(q;q)_{n}}\sum_{k=0}^n \begin{bmatrix} n \\k \end{bmatrix} \Phi_{n-k}^{(a)}(x,1|q)
(-x)^k q^{\binom k2} \D_x^k f(x).
\end{align}
We immediately complete the proof by comparing the power series of $t^n$ on both sides of the above equation.
\end{proof}

\begin{theorem}\label{3th2}
If $f(x,a)$ is a function of two variables $x$ and $a$, then, under suitable convergence conditions, we have
\begin{align}
\Delta_{x,a}^n f(x,a)&=\sum_{k=0}^n \begin{bmatrix} n \\k \end{bmatrix}
(-1)^k q^{\binom k2}x^{k} \sum_{j=0}^{n-k} \begin{bmatrix} n-k \\j \end{bmatrix} (a;q)_{j}x^{j}
 \eta_a^j \D_x^k f(x,a), \label{3eq2} \\
\Omega_{x,a}^n f(x,a)&=\sum_{k=0}^n \begin{bmatrix} n \\k \end{bmatrix}
(-1)^k q^{\binom j2}  x^k (a;q)_{k}
\sum_{j=0}^{n-k} \begin{bmatrix} n-k \\j \end{bmatrix} (aq^{k};q)_{j}x^{n-k-j}
\eta_a^{j+k} \D_x^k f(x,a). \label{3eq3}
\end{align}
In particular, if $f$ is a function of $x$, we have
\begin{align*}
\Delta_{x,a}^n f(x)&=\sum_{k=0}^n \begin{bmatrix} n \\k \end{bmatrix}
(-x)^k q^{\binom k2} \Phi_{n-k}^{(a)}(x,1|q) \D_x^k f(x), \\
\Omega_{x,a}^n f(x)&=\sum_{k=0}^n \begin{bmatrix} n \\k \end{bmatrix}
(-1)^k q^{\binom k2}  x^k (a;q)_k \Phi_{n-k}^{(aq^k)}(1,x|q)\D_x^k f(x).
\end{align*}
\end{theorem}
\begin{proof}
Appealing to (\ref{eq6}) and (\ref{2eq3}), we obtain
\begin{align*}
\exp_q(t\Delta_{x,a})f(x)
=&\sum_{k=0}^{\infty}\frac{t^{k}}{(q;q)_{k}}
\Delta_{x,a}^{k}f(x,a)\\
=&\sum_{k=0}^{\infty}\frac{t^{k}}{(q;q)_{k}}
\sum_{n=0}^{k}\begin{bmatrix} k \\n \end{bmatrix}
(a;q)_{n}x^{n}\eta_{a}^{n}f(xq^{k-n},a)\\
=&\sum_{j=0}^{\infty}\frac{t^{j}}{(q;q)_{j}} \sum_{n=0}^{\infty}\frac{(a;q)_{n}}{(q;q)_{n}}
(xt)^{n}\eta_{a}^{n}
\sum_{k=0}^{\infty}\frac{(-xt)^{k}}{(q;q)_{k}}
q^{\binom k2}\D_{x}^{k}f(x,a)\\
=&\sum_{n=0}^{\infty}\frac{t^{n}}{(q;q)_{n}}
\sum_{j=0}^{n}\begin{bmatrix} k \\n \end{bmatrix}(a;q)_{j}x^{j}\eta_{a}^{j}
\sum_{k=0}^{\infty}\frac{(-xt)^{k}}{(q;q)_{k}}
q^{\binom k2}\D_{x}^{k}f(x,a)\\
=&\sum_{n=0}^{\infty}\frac{t^{n}}{(q;q)_{n}}
\sum_{k=0}^n \begin{bmatrix} n \\k \end{bmatrix}
(-x)^k q^{\binom k2} \sum_{j=0}^{n-k} \begin{bmatrix} n-k \\j \end{bmatrix} (a;q)_{j}x^{j}
 \eta_a^j \D_x^k f(x,a).
\end{align*}
Comparing the coefficients of $t^{n}$ in the above equation, this completes the proof of the first equation.
Taking advantage of (\ref{eq6}) and (\ref{eq2}), we find that
\begin{align*}
\Omega_{x,a}^n f(x,a)&=\sum_{k=0}^n \begin{bmatrix} n \\k \end{bmatrix}
(a;q)_k x^{n-k} \eta_a^k f(xq^k,a) \\
&=\sum_{k=0}^n \begin{bmatrix} n \\k \end{bmatrix}
(a;q)_k x^{n-k} \eta_a^k  \sum_{j=0}^k \begin{bmatrix} k \\j \end{bmatrix} (-1)^j q^{\binom j2} x^j \D_x^j f(x,a) \\
&=\sum_{j=0}^n \sum_{k=0}^{n-j} \begin{bmatrix} n \\j \end{bmatrix}
\begin{bmatrix} n-j \\k \end{bmatrix} x^{n-j-k} (a;q)_{k+j} (-1)^j
q^{\binom j2} x^j \eta_a^{k+j} \D_x^j f(x,a) \\
&=\sum_{j=0}^n \begin{bmatrix} n \\j \end{bmatrix} (-1)^j q^{\binom j2} x^j
(a;q)_j \sum_{k=0}^{n-j} \begin{bmatrix} n-j \\k \end{bmatrix}
x^{n-j-k} (aq^j;q)_k \eta_a^{k+j} \D_x^j f(x,a),
\end{align*}
which is equivalent to the right-hand side of equation in Theorem \ref{3th2}.
\end{proof}

\begin{theorem}\label{Th2}
If $f(x,y,a,b)$ is a function of four variables, then, under suitable convergence conditions, we have the following $q$-exponential identity
\begin{align*}
\exp&_q(t\Delta_{y,b}\Delta_{x,a})f(x,y,a,b) \\
=&\sum_{s,k,l,n=0}^{\infty} \frac{ (a;q)_{s+k} (b;q)_{s+l} x^{s+k} y^{s+l}
t^{s+k+l+n} q^{kl} } {(q;q)_s (q;q)_k (q;q)_l (q;q)_n }
f(xq^{n+l},yq^{n+k},aq^k,bq^{s+l}).
\end{align*}
\end{theorem}
\begin{proof}
It follows from (\ref{eq6}) that
\begin{align}\label{2th3}
\exp_q(t\Delta_{x,a})f(x,y,a,b)
&=\s \frac{t^n}{(q;q)_n} \Delta_{x,a}^n f(x,y,a,b)\nonumber\\
&=\s \frac{t^n}{(q;q)_n} \sum_{k=0}^n \begin{bmatrix} n \\k \end{bmatrix} (a;q)_k x^k \eta_a^k \eta_x^{n-k} f(x,y,a,b) \nonumber\\
&=\sum_{k=0}^{\infty} \frac{(a;q)_k (xt)^k}{(q;q)_k}
\s \frac{t^n}{(q;q)_n} f(xq^n,y,aq^k,b).
\end{align}
Letting $t\to t\Delta_{y,b}$ in the above equation, we obtain
\begin{align}\label{2eq6}
 \exp_q(t\Delta_{y,b}\Delta_{x,a})f(x,y,a,b)
=\sum_{k=0}^{\infty} \frac{(a;q)_k (xt)^k}{(q;q)_k} \Delta_{y,b}^k
\s \frac{t^n}{(q;q)_n} \Delta_{y,b}^n f(xq^n,y,aq^k,b).
\end{align}
As a matter of fact, by (\ref{eq6}) and simple calculation, we have
\begin{align}\label{2eq7}
\s \frac{t^n}{(q;q)_n} \Delta_{y,b}^n f(xq^n,y,aq^k,b)
&=\s \frac{t^n}{(q;q)_n} \sum_{l=0}^n \begin{bmatrix} n \\l \end{bmatrix} (b;q)_l y^l \eta_b^l \eta_y^{n-l} f(xq^n,y,aq^k,b) \nonumber \\
&=\s \frac{t^n}{(q;q)_n} \sum_{l=0}^n \begin{bmatrix} n \\l \end{bmatrix} (b;q)_l y^l  f(xq^n,yq^{n-l},aq^k,bq^l) \nonumber \\
&=\sum_{l=0}^{\infty} \frac{(b;q)_l (yt)^l}{(q;q)_l} \s
\frac{t^n}{(q;q)_n} f(xq^{n+l},yq^n,aq^k,bq^l).
\end{align}
Furthermore, suppose $g$ is a function of variables $y$ and $b$, we have
\begin{align}\label{2eq8}
\sum_{k=0}^{\infty} \frac{(a;q)_k (xt)^k}{(q;q)_k}\Delta_{y,b}^k g(y,b)
&= \sum_{k=0}^{\infty} \frac{(a;q)_k (xt)^k}{(q;q)_k}
\sum_{s=0}^k \begin{bmatrix} k \\s \end{bmatrix} (b;q)_s y^s
\eta_b^s \eta_y^{k-s} g(y,b) \nonumber \\
&=\sum_{s=0}^{\infty} \frac{(a,b;q)_s}{(q;q)_s} (xyt)^s \sum_{k=0}^{\infty}
\frac{(aq^s;q)_k}{(q;q)_k}(xt)^k g(yq^k,bq^s).
\end{align}
Finally, using the formulas (\ref{2eq7}) and (\ref{2eq8}), (\ref{2eq6}) can be written as
\begin{align*}
 \exp_q&(t\Delta_{y,b}\Delta_{x,a})f(x,y,a,b)\\
=&\sum_{s=0}^{\infty} \frac{(a,b;q)_s}{(q;q)_s} (xyt)^s \sum_{k=0}^{\infty}
 \frac{(aq^s;q)_k}{(q;q)_k}(xt)^k
 \s \frac{t^n}{(q;q)_n} \Delta_{y,b}^n
 f(xq^{n},yq^{k},aq^k,bq^{s}) \\
=&\sum_{s=0}^{\infty} \frac{(a,b;q)_s}{(q;q)_s} (xyt)^s \sum_{k=0}^{\infty}
\frac{(aq^s;q)_k}{(q;q)_k}(xt)^k \sum_{l=0}^{\infty} \frac{(bq^s;q)_l (ytq^k)^l}{(q;q)_l} \s \frac{t^n}{(q;q)_n} f(xq^{n+l},yq^{n+k},aq^k,bq^{s+l}),
\end{align*}
which is equivalent to the right-hand side of the theorem. The proof is completed.
\end{proof}
Using a similar method as Theorem \ref{Th2}, we can also obtain the following Theorem \ref{2eq5}.
\begin{theorem}\label{2eq5}
If $f(x,y,a,b)$ is a function of four variables, then, under suitable convergence conditions, we have the following $q$-exponential identity
\begin{align*}
\exp&_q(t\Omega_{x,a}\Omega_{y,b})f(x,y,a,b) \nonumber\\
=&\sum_{s,k,l,n=0}^{\infty} \frac{ (a;q)_{s+k} (b;q)_{s+l} x^{l+n} y^{k+n}
t^{s+k+l+n} q^{sn} } {(q;q)_s (q;q)_k (q;q)_l (q;q)_n }
f(xq^k,yq^{l+s},aq^k,bq^{l+s}).
\end{align*}
\end{theorem}

\section{Identities for the homogeneous Hahn polynomials}\label{Sec3}

Al-Salam and Carlitz \cite{Al1965} found two bilinear generating functions by the transformation theory of $q$-series, which is also called the $q$-Mehler formula for the homogeneous Hahn polynomials. This section will study the homogeneous Hahn polynomials from the perspective of operators $\Delta_{x,a}$ and $\Omega_{x,a}$. For example, we give a new proof of bilinear generating function for Hahn polynomials (see Theorems \ref{3th1} and \ref{3th4}). Taking $f\equiv1$ in Theorem \ref{Th2}, we obtain following Theorem \ref{Th3}.
\begin{theorem}\label{Th3}
For $\max \{|t|,|xt|,|yt|\}<1$, we have the following operator identity
\begin{align*}
\exp_q(t\Delta_{y,b}\Delta_{x,a})1
=\frac{(axt,byt;q)_{\infty}}{(t,xt,yt;q)_{\infty}}
{}_3\phi_2\left( \begin{gathered} a, b, t \\ axt, byt \end{gathered}
 ;\,q, xyt \right).
\end{align*}
\end{theorem}

\begin{proof}
Substituting $f(x,y,a,b)\equiv1$ in Theorem \ref{Th2}, we have
\begin{align*}
 \exp_q(t\Delta_{y,b}\Delta_{x,a})1
&=\sum_{s,k,l,n=0}^{\infty} \frac{ (a;q)_{s+k} (b;q)_{s+l} x^{s+k} y^{s+l}
  t^{s+k+l+n} q^{kl} } {(q;q)_s (q;q)_k (q;q)_l (q;q)_n } \\
&=\frac{1}{(t;q)_{\infty}} \sum_{s,k=0}^{\infty} \frac{ (a;q)_{s+k}
  (b;q)_s (xt)^{s+k} y^s  } {(q;q)_s (q;q)_k   }
  \sum_{l=0}^{\infty} \frac{(bq^s;q)_l}{(q;q)_l} (ytq^k)^l \\
&=\frac{(byt;q)_{\infty}}{(t,yt;q)_{\infty}}  \sum_{s,k=0}^{\infty}
 \frac{ (a;q)_{s+k} (b;q)_s (yt;q)_k (xt)^{s+k} y^s  }
 {(q;q)_s (q;q)_k (byt;q)_{s+k}  } \\
&=\frac{(byt;q)_{\infty}}{(t,yt;q)_{\infty}} \s \frac{(a;q)_n (xt)^n}{(q;q)_{n}(byt;q)_n}
\sum_{s=0}^n \begin{bmatrix} n \\s \end{bmatrix}
(b;q)_{s}(yt;q)_{n-s}y^{s}.
\end{align*}
Taking $\alpha\to b$, $\beta \to yt$, $u\to y$ and $v\to 1$ into \cite[Lemma 8.1]{Liu2022}, then we have
\begin{align*}
\exp_q(t\Delta_{y,b}\Delta_{x,a})1
=\frac{(byt;q)_{\infty}}{(t,yt;q)_{\infty}} \s \frac{(a;q)_n \Phi_n^{(b,yt)}(y,1|q)(xt)^n}
{(q;q)_n (byt;q)_n} .
\end{align*}
Letting $t\to xt$, $\gamma\to a$, $\alpha\to b$, $\beta\to yt$, $u\to y$ and $v\to 1$ in \cite[Lemma 8.1]{Liu2022}. The above equation can be written as
\begin{align*}
\exp_q(t\Delta_{y,b}\Delta_{x,a})1=
\frac{(a,byxt,xyt^2;q)_{\infty}}{(t,xt,yt,xyt;q)_{\infty}}
{}_3\phi_2\left( \begin{gathered} byt/a, xyt, xt \\ bxyt, xyt^2 \end{gathered} ;\,q, a \right).
\end{align*}
Further, letting $a_1\to byt/a$, $a_2\to xyt$, $a_3\to xt$, $b_1\to bxyt$, $b_2\to xyt^2$ in \cite[Proposition 7.3]{Liu2022}. The above equation is equivalent to
\begin{align*}
\exp_q(t\Delta_{y,b}\Delta_{x,a})1=
\frac{(axt,bxyt;q)_{\infty}}{(t,xt,xyt;q)_{\infty}}
{}_3\phi_2\left( \begin{gathered} ax, xt, b \\ axt, bxyt
\end{gathered};\,q, yt\right).
\end{align*}
Using \cite[Proposition 7.3]{Liu2022} again on the right side of the above equation yields
\begin{align*}
\exp_q(t\Delta_{y,b}\Delta_{x,a})1&=
\frac{(axt,bxyt;q)_{\infty}}{(t,xt,xyt;q)_{\infty}}
\frac{(xyt, byt;q)_{\infty}}{(bxyt, yt;q)_{\infty}}
{}_3\phi_2\left( \begin{gathered} a, b, t \\ axt, byt
\end{gathered};\,q, xyt\right) \\
&=\frac{(axt,byt;q)_{\infty}}{(t,xt,yt;q)_{\infty}}
{}_3\phi_2\left( \begin{gathered} a, b, t \\ axt, byt \end{gathered}
 ;\,q, xyt \right).
\end{align*}
This completes the proof.
\end{proof}

\begin{theorem}\label{3th1}
For $\max \{|t|,|xt|,|yt|\}<1$, we have
\begin{align*}
\s \Phi_n^{(a)}(x,1|q) \Phi_n^{(b)}(y,1|q) \frac{t^n}{(q;q)_n}
=\frac{(axt,byt;q)_{\infty}}{(t,xt,yt;q)_{\infty}}
{}_3\phi_2\left( \begin{gathered} a, b, t \\ axt, byt \end{gathered}
 ;\,q, xyt \right),
\end{align*}
which is called the $q$-Mehler formula for the Hahn polynomials.
\end{theorem}
\begin{proof}
Combining Proposition \ref{proposition3} and Theorem \ref{Th3}, we immediately complete the proof.
\end{proof}

\begin{theorem}\label{3th1'}
For $\max \{|xt|,|yt|,|xyt|\}<1$, we have
\begin{align*}
\s \Phi_n^{(a)}(1,x|q) \Phi_n^{(b)}(1,y|q) \frac{t^n}{(q;q)_n}
=\frac{(aty,btx;q)_{\infty}}{(xt,yt,xyt;q)_{\infty}}
{}_3\phi_2\left( \begin{gathered} a, b, xyt \\ aty, btx \end{gathered}
 ;\,q, t \right).
\end{align*}
\end{theorem}

\begin{proof}
Taking $f(x,y,a,b)\equiv1$ in Theorem \ref{2eq5}, we have
\begin{align}\label{eq7}
\exp_q(t\Omega_{x,a} \Omega_{y,b})1
&=\sum_{s,k,l,n=0}^{\infty} \frac{ (a;q)_{s+k} (b;q)_{s+l} x^{l+n} y^{k+n}
t^{s+k+l+n} q^{sn} } {(q;q)_s (q;q)_k (q;q)_l (q;q)_n } \nonumber\\
&=\sum_{s,k,l=0}^{\infty} \frac{ (a;q)_{s+k} (b;q)_{s+l} x^l y^k
t^{s+k+l} } {(q;q)_s (q;q)_k (q;q)_l} \s \frac{(xytq^s)^n}{(q;q)_n}
\nonumber \\
&=\frac{1}{(xyt;q)_{\infty}} \sum_{s=0}^{\infty} \frac{(a,b,xyt;q)_s t^s}{(q;q)_s} \sum_{k=0}^{\infty} \frac{(aq^s;q)_k (yt)^k}{(q;q)_k}
\sum_{l=0}^{\infty} \frac{(bq^s;q)_l (xt)^l}{(q;q)_l} \nonumber\\
&=\frac{(aty,btx;q)_{\infty}}{(xt,yt,xyt;q)_{\infty}}
{}_3\phi_2\left( \begin{gathered} a, b, xyt \\ aty, btx \end{gathered}
 ;\,q, t \right).
\end{align}
Combining Proposition \ref{proposition3} and  (\ref{eq7}). This completes the proof.
\end{proof}
We should point out that Theorems \ref{3th1}-\ref{3th1'} are special cases of \cite[Theorem 3.2]{L2015}. However, the following Theorems \ref{3th3}-\ref{3th4} can be regarded as the generalization of Theorems \ref{3th1}-\ref{3th1'}, respectively. The authors have not found them in other literature.

\begin{theorem}\label{3th3}
For $\max\{|t|,|xt|,|yt|\}<1$ and any non-negative integers $m$ and $n$, we have
\begin{align*}
 \sum_{k=0}^{\infty} &\Phi_{n+k}^{(a)}(x,1|q) \Phi_{m+k}^{(b)}(y,1|q) \frac{t^k}{(q;q)_k} \\
=&\frac{(axt,byt;q)_{\infty}}{(t,xt,yt;q)_{\infty}}
\sum_{k=0}^m \sum_{j=0}^n \begin{bmatrix} m \\k \end{bmatrix} \begin{bmatrix} n \\j \end{bmatrix} (xt;q)_{n-j}(yt;q)_{m-k} x^{j} y^{k}\\
&\times \s\frac{(a;q)_{ j+l}(b;q)_{k+l}(t;q)_{l}}
{(q;q)_{l}(axt;q)_{n+l}(byt;q)_{m+l}}(xytq^{m+n-k-j})^{l}.
\end{align*}
\end{theorem}

\begin{proof}
It follows from (\ref{2eq4}) that
\begin{align}\label{3Eq5}
\sum_{k=0}^{\infty}\frac{t^{k}}{(q;q)_{k}}
\Delta_{y,b}^{m}\Delta_{x,a}^{n}1
=\Delta_{y,b}^m \Delta_{x,a}^n \exp_q(t\Delta_{x,a} \Delta_{y,b})1.
\end{align}
On the other hand, taking Theorem \ref{Th3} and (\ref{eq6}) into the right-hand side of the above equation, we find that
\begin{align*}
\Delta_{y,b}^m& \Delta_{x,a}^n \exp_q(t\Delta_{x,a} \Delta_{y,b})1 \\
=&\sum_{k=0}^m \sum_{j=0}^n \begin{bmatrix} m \\k \end{bmatrix} \begin{bmatrix} n \\j \end{bmatrix} (a;q)_j(b;q)_k x^{j} y^{k} \eta_a^j \eta_x^{n-j} \eta_b^k \eta_y^{m-k}
\exp_q(t\Delta_{x,a} \Delta_{y,b})1 \\
=&\sum_{k=0}^m \sum_{j=0}^n \begin{bmatrix} m \\k \end{bmatrix} \begin{bmatrix} n \\j \end{bmatrix} (a;q)_j(b;q)_k x^{j} y^{k} \\ &\times\frac{(axtq^{n},bytq^{m};q)_{\infty}}
{(t,xtq^{n-j},ytq^{m-k};q)_{\infty}}
\sum_{}{} \frac{(aq^{j},bq^{k},t;q)_l}{(q,axtq^{n},bytq^{m};q)_l}
(xytq^{m+n-k-j})^n \\
=&\frac{(axt,byt;q)_{\infty}}{(t,xt,yt;q)_{\infty}}
\sum_{k=0}^m \sum_{j=0}^n \begin{bmatrix} m \\k \end{bmatrix} \begin{bmatrix} n \\j \end{bmatrix} (xt;q)_{n-j}(yt;q)_{m-k} x^{j} y^{k}\\
&\times \s\frac{(a;q)_{ j+l}(b;q)_{k+l}(t;q)_{l}}
{(q;q)_{l}(axt;q)_{n+l}(byt;q)_{m+l}}(xytq^{m+n-k-j})^{l}.
\end{align*}
Combining the above identity and (\ref{3Eq5}), then we complete the proof of Theorem \ref{3th3}.
\end{proof}

Using a similar method as the above theorem, we can also obtain following Theorem \ref{3th4}.
\begin{theorem}\label{3th4}
For $\max\{|xt|,|yt|,|xyt|\}<1$ and any non-negative integers $m$ and $n$, we have
\begin{align*}
& \sum_{k=0}^{\infty} \Phi_{n+k}^{(a)}(1,x|q) \Phi_{m+k}^{(b)}(1,y|q) \frac{t^k}{(q;q)_k} \\
&=\frac{(aty,btx;q)_{\infty}}{(xt,yt,xyt;q)_{\infty}}
\sum_{k=0}^{m} \sum_{j=0}^n \begin{bmatrix} m \\k \end{bmatrix}
\begin{bmatrix} n \\j \end{bmatrix} (xt;q)_j (yt;q)_k
x^{n-j} y^{m-k} \\
&\quad \times \sum_{l=0}^{\infty} \frac{(a;q)_{j+l} (b;q)_{k+l} (xyt;q)_{k+j+l} t^l}
{(q;q)_l (aty;q)_{k+j+l} (btx;q)_{k+j+l}}.
\end{align*}
\end{theorem}

\begin{theorem}
For any non-negative integers $m$ and $n$, we have
\begin{align*}
\Phi_{m+n}^{(a)}(x,1|q)&=\sum_{k=0}^{\min(m,n)}
\begin{bmatrix} n\\k \end{bmatrix} \begin{bmatrix} m\\k \end{bmatrix}
(a;q)_k (q;q)_k (-1)^k q^{\binom k2}x^{k} \Phi_{n-k}^{(a)}(x,1|q)
\Phi_{m-k}^{(aq^k)}(x,1|q),\\
\Phi_{m+n}^{(a)}(1,x|q)&=\sum_{k=0}^{\min(m,n)}
\begin{bmatrix} n\\k \end{bmatrix} \begin{bmatrix} m\\k \end{bmatrix}
(a;q)_k^2 (q;q)_k (-1)^k q^{\binom k2} x^k \Phi_{n-k}^{(aq^k)}(1,x|q)
\Phi_{m-k}^{(aq^k)}(1,x|q).
\end{align*}
\end{theorem}
\begin{proof}
The proof of the second equality in the theorem is similar to the first one. We only prove the first one. Taking $f(x)=\Phi_{m}^{(a)}(x,1|q)$ into Theorem \ref{3th2} and by Proposition \ref{2eq10}, we deduce that
\begin{align}\label{eq9}
\Phi_{m+n}^{(a)}(x,1|q)
=\sum_{k=0}^n \begin{bmatrix} n \\k \end{bmatrix}
(-1)^k q^{\binom k2}\Phi_{n-k}^{(a)}(x,1|q)
 \D_x^k \Phi_m^{(a)}(x,1|q).
\end{align}
Next, by (cf. \cite[p.484]{L2015}), we have
\begin{align*}
\D_x {\Phi_n^{(\alpha)}(x,y|q)}
=\sum_{k=1}^{n}\begin{bmatrix} n\\k \end{bmatrix}
(\alpha;q)_{k}(1-q^{k})x^{k-1}y^{n-k}.
\end{align*}
Taking $y=1$ in the above identity, and then by simple calculation, we obtain
\begin{align*}
\D_x \Phi_m^{(a)}(x,1|q)
=\sum_{k=1}^m \begin{bmatrix} m \\k \end{bmatrix}
 (a;q)_k (1-q^k) x^{k-1}
=(1-a)(1-q^m)\Phi_{m-1}^{(aq)}(x,1|q).
\end{align*}
Thus, we have
\begin{align*}
\D_x^{k} \Phi_m^{(a)}(x,1|q)
=\frac{(a;q)_k (q;q)_m}{(q;q)_{m-k}} \Phi_{m-k}^{(aq^k)}(x,1|q).
\end{align*}
Substituting the above equation into (\ref{eq9}), we complete the proof.
\end{proof}

\section{Applications of operators $\Delta_{x,a}$ and $\Omega_{x,a}$}\label{Sec4}
This section mainly gives some applications of operators $\Delta_{x,a}$ and $\Omega_{x,a}$. In order to prove Heine's second transformation formula (see Theorem \ref{4th2}), we first give the following theorem.
\begin{theorem}\label{2th1}
For $\max \{|axt|,|axs| \}<1$, we have the transformation formula
\begin{align*}
{}_2\phi_2\left( \begin{gathered} xs, s \\ asx, xts \end{gathered}
 ;\,q, axt \right)=\frac{(axt;q)_{\infty} }{(axs;q)_{\infty}}
 {}_2\phi_2\left( \begin{gathered} xt, t \\ axt, xts \end{gathered}
 ;\,q, axs \right).
\end{align*}
\end{theorem}
\begin{proof}
It follows from Proposition \ref{proposition1} that
\begin{align}\label{eq5}
\exp_{q}(s\Delta_{x,a}) \exp_{q}(t\Delta_{x,a})1
&=\exp_{q}(s\Delta_{x,a}) \left\{ \frac{(axt;q)_{\infty}}
  {(t,xt;q)_{\infty}} \right\}.
\end{align}
On the other hand, by (\ref{2th3}), we have
\begin{align*}
\exp_{q}(s\Delta_{x,a})f(x,a)
=\sum_{k=0}^{\infty} \frac{(a;q)_k}{(q;q)_k} (xs)^k
\s \frac{s^n}{(q;q)_n}
f(xq^n,aq^k).
\end{align*}
Replacing (\ref{eq5}) in the above operator equation, we obtain
\begin{align*}
\exp_{q}(s\Delta_{x,a}) \left\{ \frac{(axt;q)_{\infty}}
{(t,xt;q)_{\infty}} \right\}
&=\sum_{k=0}^{\infty} \frac{(a;q)_k}{(q;q)_k} (xs)^k \s \frac{s^n}{(q;q)_n}
\frac{(axtq^{n+k};q)_{\infty}} {(t,xtq^n;q)_{\infty}} \\
&=\frac{(axt;q)_{\infty}} {(t,xt;q)_{\infty}} \sum_{k=0}^{\infty} \frac{(a;q)_k}{(q;q)_k} (xs)^k \s \frac{(xt;q)_n s^n}{(q;q)_n(axt;q)_{n+k}}\\
&=\frac{(axt;q)_{\infty}} {(t,xt;q)_{\infty}} \s \frac{s^n}{(q,axt;q)_n}
\sum_{k=0}^n  \begin{bmatrix} n \\k \end{bmatrix} (a;q)_k (xt;q)_{n-k} x^k\\
&=\frac{(axt;q)_{\infty}} {(t,xt;q)_{\infty}} \s
\frac{\Phi_n^{(a,xt)}(x,1|q)}{(q,axt;q)_n} s^n,
\end{align*}
where
\begin{align*}
\Phi_n^{(\alpha,\beta)}(u,v|q)=\sum_{k=0}^n \begin{bmatrix} n \\k \end{bmatrix} (\alpha;q)_k (\beta;q)_{n-k} u^k v^{n-k}
\end{align*}
has been defined in \cite[(8.2)]{Liu2022}. Then we have
\begin{align}\label{eq10}
\exp_{q}(s\Delta_{x,a}) \exp_{q}(t\Delta_{x,a})1
=\frac{(axt;q)_{\infty}} {(t,xt;q)_{\infty}} \s
\frac{\Phi_n^{(a,xt)}(x,1|q)}{(q,axt;q)_n} s^n.
\end{align}
By \cite[Lemma 8.1]{Liu2022}, we know that
\begin{align*}
\s \frac{(\gamma;q)_n \Phi_n^{(\alpha,\beta)}(u,v|q)t^n}{(q,\alpha \beta;q)_n}=\frac{(\gamma,\alpha ut,\beta vt;q)_{\infty}}{(\alpha\beta,tu,tv;q)_{\infty}}
{}_3\phi_2\left( \begin{gathered} \alpha\beta/\gamma, tu, tv \\ \alpha ut, \beta vt \end{gathered} ;\,q, \gamma \right).
\end{align*}
Letting $\gamma \to 0$, and further letting $\alpha\to a$, $\beta\to xt$, $u\to x$ and $v\to1$ in the above equation yields
\begin{align*}
\s \frac{\Phi_n^{(a,xt)}(x,1|q)}{(q,axt;q)_n} s^n
=\frac{(axs,xts;q)_{\infty}}{(axt,sx,s;q)_{\infty}}
{}_2\phi_2\left( \begin{gathered} xs, s \\ axs, xts \end{gathered}
;\,q, axt \right).
\end{align*}
Replacing the above equation into (\ref{eq10}), we obtain
\begin{align}\label{eq3}
\exp_{q}(s\Delta_{x,a}) \exp_{q}(t\Delta_{x,a})1
=\frac{(axs,xts;q)_{\infty}}{(t,s,xs,xt;q)_{\infty}}
{}_2\phi_2\left( \begin{gathered} xs, s \\ axs, xts \end{gathered}
;\,q, axt \right).
\end{align}
Using the symmetry properties of the left hand side $s$ and $t$ of operator equation (\ref{eq3}), we can clearly see that
\begin{align}\label{eq4}
\exp_{q}(s\Delta_{x,a}) \exp_{q}(t\Delta_{x,a})1
=\frac{(axt,xts;q)_{\infty}}{(t,s,xs,xt;q)_{\infty}}
{}_2\phi_2\left( \begin{gathered} xt, t \\ axt, xts \end{gathered}
;\,q, axs \right).
\end{align}
Combining (\ref{eq3}) and (\ref{eq4}), we immediately deduce the conclusion.
\end{proof}

\begin{theorem}\label{4th2}
For $\max \{|c|,|z|,|c/b|\}<1$, the famous Heine's second transformation formula (cf. \cite[(1.4.5)]{GR}) as follows
\begin{align}\label{Heine}
{}_2\phi_1\left( \begin{gathered} a, b \\ c \end{gathered}
;\,q, z \right)=\frac{(c/b, bz;q)_{\infty}}{(c,z;q)_{\infty}}
{}_2\phi_1\left( \begin{gathered} abz/c, b \\ bz \end{gathered}
;\,q, c/b \right).
\end{align}
\end{theorem}

\begin{proof}
Applying Jackson transformation formula (\ref{1eq3}) to both sides of the equation in Theorem \ref{2th1}. Using simple calculations, we have
\begin{align*}
{}_2\phi_1\left( \begin{gathered} xs, ax \\ axs \end{gathered}
;\,q, t \right)=\frac{(s,axt;q)_{\infty}}{(t,axs;q)_{\infty}}
{}_2\phi_1\left( \begin{gathered} xt, ax \\ axt \end{gathered}
;\,q, s \right).
\end{align*}
Letting $sx\to a$, $ax\to b$, $axs\to c$ and $t\to z$ in the above equation, this completes the proof of Theorem \ref{Heine}.
\end{proof}
\begin{remark}
Applying the operator $\exp(s\Omega_{x,a}) \exp(t\Omega_{x,a})$ to constant $1$, and use the method similar to the Theorem \ref{2th1}, we can obtain the following equation
\begin{align*}
{}_2\phi_1\left( \begin{gathered} xt, a\\ at \end{gathered}
 ;\,q, s \right)=\frac{(as,t;q)_{\infty} }{(at,s;q)_{\infty}}
 {}_2\phi_1\left( \begin{gathered} xs, a \\ as \end{gathered}
 ;\,q, t \right).
\end{align*}
Letting $xt\to a$, $a\to b$, $at\to c$ and $s\to z$ in the above equation, it can also obtain (\ref{Heine}).
\end{remark}
In 1965, Al-Salam and Carlit defined the following polynomial (cf. \cite[(1.15)]{Al1965}), that is,
\begin{align*}
\psi_n^{(a)}(x)=\sum_{r=0}^n (-1)^r \begin{bmatrix} n \\r \end{bmatrix}
q^{\binom {r+1}{2}-nr} (a^{-1};q)_r (ax)^r.
\end{align*}
The following Theorem \ref{4th1} shows an identity of operator $\Delta_{x,a}$ involving polynomial $\psi_n^{(a)}(x)$.
\begin{theorem}\label{4th1}
For any non-negative integer $n$, we have
\begin{align*}
\sum_{k=0}^n \begin{bmatrix} n \\k \end{bmatrix} (-1)^k q^{\binom k2} \psi_k^{(a)}(x) \Delta_{x,a}^{n-k} f(x)=(-x)^n q^{\binom n2} \D_x^n f(x),
\end{align*}
In particular, for any positive integer $n$, we have
\begin{align}\label{5eq5}
\sum_{k=0}^n \begin{bmatrix} n \\k \end{bmatrix} (-1)^k q^{\binom k2} \psi_k^{(a)}(x) \Phi_{n-k}^{(a)}(x,1|q)=0.
\end{align}
\end{theorem}
\begin{proof}
It follows from (\ref{2eq11}) that
\begin{align}\label{5eq2}
\s \frac{t^n}{(q;q)_n} \Delta_{x,a}^n f(x)
&=\s \frac{t^n}{(q;q)_n} \sum_{k=0}^n \begin{bmatrix} n \\k \end{bmatrix} (-x)^k q^{\binom k2} \Phi_{n-k}^{(a)}(x,1|q) \D_x^k f(x) \nonumber \\
&=\sum_{k=0}^{\infty} \sum_{n=k}^{\infty} \frac{(-x)^k q^{\binom k2} t^n}
{(q;q)_k (q;q)_{n-k}} \Phi_{n-k}^{(a)}(x,1|q) \D_x^k f(x) \nonumber \\
&=\sum_{k=0}^{\infty} \frac{(-xt)^k q^{\binom k2}}{(q;q)_k} \s
\frac{t^n}{(q;q)_n} \Phi_n^{(a)}(x,1|q) \D_x^k f(x) \nonumber \\
&=\frac{(axt;q)_{\infty}}{(t,xt;q)_{\infty}} \sum_{k=0}^{\infty} \frac{(-xt)^k q^{\binom k2}}{(q;q)_k} \D_x^k f(x).
\end{align}
In the last step, the generation function formula (\ref{2eq12}) is used.
We should notice that (cf. \cite[(1.16)]{Al1965})
\begin{align}\label{5eq1}
\s (-1)^n q^{\binom n2} \psi_n^{(a)}(x) \frac{w^n}{(q;q)_n}
=\frac{(w,wx;q)_{\infty}}{(axw;q)_{\infty}}.
\end{align}
Letting $w\to t$ in (\ref{5eq1}), then replacing it in (\ref{5eq2}), we obtain
\begin{align*}
\s \frac{(-xt)^n q^{\binom n2}}{(q;q)_n} \D_x^n f(x)
&=\s (-1)^n q^{\binom n2} \psi_n^{(a)}(x) \frac{t^n}{(q;q)_n}
\s \frac{t^n}{(q;q)_n} \Delta_{x,a}^n f(x) \\
&=\s \frac{t^n}{(q;q)_n} \sum_{k=0}^n \begin{bmatrix} n \\k \end{bmatrix}
(-1)^k q^{\binom k2} \psi_k^{(a)}(x) \Delta_{x,a}^{n-k} f(x).
\end{align*}
Comparing the power series of $t^n$ in the above equation yields
\begin{align*}
\sum_{k=0}^n \begin{bmatrix} n \\k \end{bmatrix} (-1)^k q^{\binom k2} \psi_k^{(a)}(x) \Delta_{x,a}^{n-k} f(x)=(-x)^n q^{\binom n2} \D_x^n f(x).
\end{align*}
Taking $f(x)\equiv1$ in the above equation, we obtain (\ref{5eq5}). This complete the proof.
\end{proof}
The following theorem presents a relation between the generalized operator $\Delta_{x,a}$ and the operator $\Delta_{x}$. For the definition of $\Delta_{x}$, see Theorem \ref{Th}.
\begin{theorem}
For $\max\{ |a|,|s|,|t|\}<1$, we have
\begin{align*}
\exp_q(a\Delta_{x})1 \exp_q(s\Delta_{x,a}) \exp_q(t\Delta_{x,a})1
&=\exp_q(t\Delta_{x}) \exp_q(s\Delta_{x}) \exp_q(a\Delta_{x})1.
\end{align*}
\end{theorem}
\begin{proof}
Letting $a\to xs$, $b\to ax$, $c\to axs$ and $z\to t$ in (\ref{1eq3}), we have
\begin{align*}
{}_2\phi_2\left( \begin{gathered} xs, s \\ axs, xts \end{gathered};\,q, axt \right)
=\frac{(t;q)_{\infty}}{(xst;q)_{\infty}}
{}_2\phi_1\left( \begin{gathered} xs, ax\\ axs \end{gathered};\,q, t \right).
\end{align*}
Applying the above equation to (\ref{eq3}), we have
\begin{align*}
\exp_q(s\Delta_{x,a}) \exp_q(t\Delta_{x,a})1
=\frac{(axs;q)_{\infty}}{{(s, xt, xs)}_{\infty}}
{}_2\phi_1 \left( \begin{gathered}
sx, ax  \\ asx \end{gathered};\,q, t \right).
\end{align*}
It follows from \cite[Proposition 7.1]{Liu2022} that
\begin{align*}
\exp_q(t\Delta_{x})1 \exp_q(s\Delta_{x}) \exp_q(a\Delta_{x})1
=\frac{(asx;q)_{\infty}}{(s,a,tx,sx,ax;q)_{\infty}}
{}_2\phi_1 \left( \begin{gathered}
sx, ax  \\ asx \end{gathered};\,q, t \right).
\end{align*}
Comparing the above two identities and by \cite[Proposition 1.12]{Liu2022}. This completes the proof.
\end{proof}

\section{A generalization of $q$-Gaussian summation}

The $q$-Gaussian summation formula is one of the most fundamental and important formulas in the theory of basic hypergeometric series. That is,
\begin{align*}
\s \frac{(a,b;q)_n}{(q,c;q)_n} \left(\frac{c}{ab}\right)^n
=\frac{(c/a,c/b;q)_{\infty}}{(c,c/ab;q)_{\infty}}, \,\,
 \left|\frac{c}{ab}\right|<1,
\end{align*}
which was first proved by Heine in 1847 (cf. \cite[p.14]{GR}). In \cite[Theorem 7.5]{Liu2022}, Liu gave an elegant generalization of $q$-Gaussian summation as follows.
\begin{theorem}\label{5th6}
For $\max\{|c/ab|,|az|\}<1$,
\begin{align*}
\s \frac{(a,b;q)_n}{(q,c;q)_n} \left(\frac{c}{ab}\right)^n
{}_2\phi_1\left( \begin{gathered} c/a,bq^n \\ cq^n \end{gathered} ;\,q, az \right)=\frac{(c/a,c/b,abz;q)_{\infty}}{(c,c/ab,az;q)_{\infty}}.
\end{align*}
\end{theorem}
We clearly see that the above theorem degenerates into the famous $q$-Gaussian summation when $z=0$. This section shall use the property of operator $\Omega_{u,a}$ to give a more general $q$-summation, which includes Theorem \ref{5th6} as a special case. For that, we first give the general double basic hypergeometric series is defined by (cf. \cite[p.284]{GR})
\begin{align*}
\Theta^{A:B;C}_{D:E;F} \left[ \begin{gathered}
a_A: b_B; c_C \\ d_D: e_E; f_F \end{gathered} ;\,q; x, y \right]
=\sum_{m=0}^{\infty} \s \frac{(a_A;q)_{m+n} (b_B;q)_m (c_C;q)_n}
{(d_D;q)_{m+n} (q,e_E;q)_m (q,f_F;q)_n} \\
\times \left[(-1)^{m+n} q^{\binom {m+n}{2}}\right]^{D-A}
\left[(-1)^m q^{\binom {m}{2}}\right]^{1+E-B}
\left[(-1)^n q^{\binom {n}{2}}\right]^{1+F-C}x^my^n,
\end{align*}
where $a_A$ abbreviates the array of $A$ parameters $a_1, a_2, \cdots, a_A$, etc.
\begin{theorem}\label{6th}
For $\max \{|t|,|x|,|s|,|atx|<1\}$, we have the $q$-summation formula
\begin{align*}
&\s \frac{(ut;q)_n x^n}{(q;q)_n}
\Theta^{2:3;1}_{2:2;0} \left( \begin{gathered}
a,utx:a,utx,utq^n; 0 \\ atx,0: atx,0; - \end{gathered} ;\,q;s, tq^n \right)\\
&=\frac{(as,tx,utx;q)_{\infty}}{(x,s,atx;q)_{\infty}}
{}_2\phi_1\left( \begin{gathered} a,us \\ as \end{gathered} ;\,q, t \right).
\end{align*}
\end{theorem}

\begin{proof}
It follows from \cite[Theorem 7.5]{Liu2022} that
\begin{align*}
\frac{1}{(x,y,yz;q)_{\infty}}=\s \frac{x^n}{(q;q)_n (yq^n,yz,xy;q)_{\infty}}.
\end{align*}
Replacing $y$ by $t\Omega_{u,a}$ in the above equation, and letting $tz=s$, we deduce that
\begin{align*}
\frac{1}{(x;q)_{\infty}} \exp_q(t\Omega_{u,a}) \exp_q(s\Omega_{u,a})=
\s \frac{x^n}{(q;q)_n}  \exp_q(s\Omega_{u,a}) \exp_q(tq^n \Omega_{u,a})
\exp_q(xt\Omega_{u,a}).
\end{align*}
Acting $f(x)\equiv1$ in the both sides of the above equation, we obtain
\begin{align}\label{6eq1}
\frac{1}{(x;q)_{\infty}} \exp_q(t\Omega_{u,a}) \exp_q(s\Omega_{u,a})1=
\s \frac{x^n}{(q;q)_n}  \exp_q(s\Omega_{u,a}) \exp_q(tq^n \Omega_{u,a})
\exp_q(xt\Omega_{u,a})1.
\end{align}
Suppose $f$ is a function of $x$ and $a$, then combined with (\ref{eq6}), we have
\begin{align}\label{5eq3}
\exp_q(t\Omega_{x,a})f(x,a)
&=\s \frac{t^n}{(q;q)_n} \Omega_{x,a}^n f(x,a)\nonumber\\
&=\s \frac{t^n}{(q;q)_n} \sum_{k=0}^n \begin{bmatrix} n \\k \end{bmatrix} (a;q)_k x^{n-k} \eta_a^k \eta_x^k f(x,a) \nonumber\\
&=\sum_{k=0}^{\infty} \frac{(a;q)_k t^k}{(q;q)_k}
\s \frac{(xt)^n}{(q;q)_n} f(xq^k,aq^k).
\end{align}
It follows from Proposition \ref{proposition1} that
\begin{align*}
\exp_{q}(t\Omega_{x,a}) \exp_{q}(s\Omega_{x,a})1
&=\exp_{q}(t\Omega_{x,a}) \left\{ \frac{(as;q)_{\infty}}
{(s,xs;q)_{\infty}} \right\}.
\end{align*}
Substituting (\ref{5eq3}) into the above operator equation, we can easily obtain
\begin{align}\label{6eq2}
\exp_q(t\Omega_{u,a})\exp_q(s\Omega_{u,a})1
=\frac{(as;q)_{\infty}}{(s,us,ut;q)_{\infty}}
{}_2\phi_1\left( \begin{gathered} a,us \\ as \end{gathered} ;\,q, t \right).
\end{align}
Then by (\ref{eq6}) and (\ref{6eq2}), the right-hand side of (\ref{6eq1}) can be written as follows
\begin{align}\label{6eq3}
&\s \frac{x^n}{(q;q)_n}\exp_q(s\Omega_{u,a}) \exp_q(tq^n \Omega_{u,a})\exp_q(xt\Omega_{u,a})1
\nonumber \\
&=\s \frac{x^n}{(q;q)_n}\exp_q(s\Omega_{u,a}) \left\{
\frac{(axt;q)_{\infty}}{(xt,utx,utq^n;q)_{\infty}}
{}_2\phi_1\left( \begin{gathered} a,utx \\ atx \end{gathered}
;\,q, tq^n \right) \right\} \nonumber \\
&=\s \frac{x^n}{(q;q)_n} \sum_{k=0}^{\infty} \frac{(a;q)_k s^k}{(q;q)_k}
\sum_{m=0}^{\infty} \frac{(us)^m}{(q;q)_m} \frac{(atxq^k;q)_{\infty}}
{(xt,utxq^k,utq^{n+k};q)_{\infty}} \sum_{l=0}^{\infty}
\frac{(aq^k,utxq^k;q)_l}{(q,atxq^k;q)_l}(tq^n)^l \nonumber \\
&=\frac{ (atx;q)_{\infty}}{(us,ut,tx,utx;q)_{\infty}}
\sum_{n=0}^{\infty}\frac{(ut;q)_{n} x^{n}}{(q;q)_{n}}
\sum_{k=0}^{\infty} \frac{(a,utx,utq^n;q)_k  s^k}{(q,atx;q)_k}
\sum_{l=0}^{\infty} \frac{(aq^k,utxq^k;q)_l}{(q,atxq^k;q)_l}(tq^n)^l.
\end{align}
In the last step, (\ref{binomialTh}) has been used. Finally, substituting operator equations
(\ref{6eq2})-(\ref{6eq3}) into (\ref{6eq1}), we clearly see that
\begin{align*}
&\frac{(as,tx,utx;q)_{\infty}}{(x,s,atx;q)_{\infty}}
{}_2\phi_1\left( \begin{gathered} a,us \\ as \end{gathered} ;\,q, t\right)\\
&=\s \frac{(ut;q)_n x^n}{(q;q)_n}
\sum_{k=0}^{\infty} \frac{(a,utx,utq^n;q)_k  s^k}{(q,atx;q)_k}
\sum_{l=0}^{\infty} \frac{(aq^k,utxq^k;q)_l}{(q,atxq^k;q)_l}(tq^n)^l.
\end{align*}
We note that the right-hand side of the above equation is equivalent to the left-hand side of the equation in Theorem \ref{6th}. This completes the proof.
\end{proof}
\begin{remark}
Taking $a=0$ in Theorem \ref{6th}, we obtain
\begin{align*}
\s \frac{(t,ut;q)_n x^n}{(q,ut^2x;q)_n}
\sum_{k=0}^{\infty} \frac{(utx,utq^n;q)_k  s^k}{(q,ut^2xq^n;q)_k}
=\frac{(tx,utx,ust;q)_{\infty}}{(x,s,ut^2x;q)_{\infty}}.
\end{align*}
Letting $s\to tz$, $t\to a$, $ut\to b$ and $ut^2x\to c$ in the above equation, which is equivalent to Theorem \ref{5th6}.
\end{remark}

\section*{Conflict of interest}
The authors declare that they have no conflict of interest.

\end{document}